\documentclass[10pt,reqno]{amsart}
\usepackage{eucal,fullpage,times,amsmath,amsthm,amssymb,mathrsfs,stmaryrd,color,enumerate,accents}
\usepackage[all]{xy}
\usepackage{url}

\newcommand{\calX}{{\mathcal{X}}}

\newcommand{\calO}{{\mathcal{O}}}

\newcommand{\calA}{{\mathcal{A}}}
\newcommand{\calB}{{\mathcal{B}}}

\newcommand{\calQ}{\mathcal{Q}}
\newcommand{\calE}{\mathcal{E}}
\newcommand{\calF}{\mathcal{F}}

\newcommand{\A}{\mathbf{A}}

\newcommand{\Z}{\mathbf{Z}}
\newcommand{\N}{\mathbf{N}}

\newcommand{\F}{\mathbf{F}}

\renewcommand{\P}{\mathbf{P}}

\newcommand{\Spec}{{\mathrm{Spec}}}

\newcommand{\Fun}{\mathrm{Fun}}

\newcommand{\Cartier}{\mathrm{Cartier}}

\newcommand{\Frob}{\mathrm{Frob}}
\newcommand{\Mod}{\mathrm{Mod}}

\newcommand{\R}{\mathrm{R}}

\newcommand{\Ch}{\mathrm{Ch}}

\newcommand{\zar}{\mathrm{Zar}}

\newcommand{\Fil}{\mathrm{Fil}}

\newcommand{\gr}{\mathrm{gr}}
\newcommand{\dR}{\mathrm{dR}}

\newcommand{\crys}{\mathrm{crys}}

\newcommand{\conj}{\mathrm{conj}}
\newcommand{\opp}{\mathrm{opp}}

\newcommand{\fram}{\mathfrak{m}}

\newcommand{\frap}{\mathfrak{p}}

\newcommand{\comment}[1]{}

\newcommand{\cosimp}[3]{\xymatrix@1{#1 \ar@<.4ex>[r] \ar@<-.4ex>[r] & {\ }#2 \ar@<0.8ex>[r] \ar[r] \ar@<-.8ex>[r] & {\ } #3 \ar@<1.2ex>[r] \ar@<.4ex>[r] \ar@<-.4ex>[r] \ar@<-1.2ex>[r] & \cdots }}
\newcommand{\colim}{\mathop{\mathrm{colim}}}

\begin{document}
\bibliographystyle{alpha}
\newtheorem{theorem}{Theorem}[section]
\newtheorem*{theorem*}{Theorem}
\newtheorem*{condition*}{Condition}
\newtheorem*{definition*}{Definition}
\newtheorem{proposition}[theorem]{Proposition}
\newtheorem{lemma}[theorem]{Lemma}
\newtheorem{corollary}[theorem]{Corollary}
\newtheorem{claim}[theorem]{Claim}

\theoremstyle{definition}
\newtheorem{definition}[theorem]{Definition}
\newtheorem{question}[theorem]{Question}
\newtheorem{remark}[theorem]{Remark}
\newtheorem{guess}[theorem]{Guess}
\newtheorem{example}[theorem]{Example}
\newtheorem{condition}[theorem]{Condition}
\newtheorem{warning}[theorem]{Warning}
\newtheorem*{notation*}{Notation}
\newtheorem{construction}[theorem]{Construction}

\title{Torsion in the crystalline cohomology of singular varieties}
\author{Bhargav Bhatt}

\begin{abstract}
This note discusses some examples showing that the crystalline cohomology of even very mildly singular projective varieties tends to be quite large. In particular, any singular projective variety with at worst ordinary double points has infinitely generated crystalline cohomology in at least two cohomological degrees. These calculations rely critically on comparisons between crystalline and derived de Rham cohomology.
\end{abstract}

\maketitle

Fix an algebraically closed field $k$ of characteristic $p > 0$ with ring of Witt vectors $W$. Crystalline cohomology is a $W$-valued cohomology theory for varieties over $k$ (see \cite{GrothendieckCrystals,Berthelot}). It is exceptionally well behaved on proper smooth $k$-varieties: the $W$-valued theory is finite dimensional \cite{BerthelotOgusBook}, and the corresponding $W[1/p]$-valued theory is a Weil cohomology theory \cite{KatzMessing} robust enough to support a $p$-adic proof of the Weil conjectures \cite{KedlayaWeil2}.

Somewhat unfortunately, crystalline cohomology is often large and somewhat unwieldy outside the world of proper smooth varieties. For example, the crystalline cohomology of a smooth affine variety of dimension $> 0$ is always infinitely generated as a $W$-module by the Cartier isomorphism (see Remark \ref{rmk:localinfiniteness}). Even worse, it is not a topological invariant: Berthelot and Ogus showed \cite[Appendix (A.2)]{BerthelotOgus} that the $0^\mathrm{th}$ crystalline cohomology group of a fat point in $\A^2$ has torsion (see also Example \ref{ex:infinitesimalcurves}). In this note, we give more examples of such unexpected behaviour:

\begin{theorem*}
Let $X$ be a proper lci $k$-variety. Then the crystalline cohomology of $X$ is infinitely generated if any of the following conditions is satisfied:
\begin{enumerate}
\item $X$ has at least one isolated toric singularity, such as a node on a curve.
\item $X$ has at least one conical singularity of low degree, such as an ordinary double point of any dimension.
\end{enumerate}
\end{theorem*}

The statement above is informal, and we refer the reader to the body of the note ---  see examples \ref{ex:nodalcurve}, \ref{ex:toricsing}, \ref{ex:odp}, and \ref{ex:curvecone}  --- for precise formulations. In contrast to the Berthelot-Ogus example, our examples are reduced and lci. We do not know whether to regard these calculations as pathological or indicative of deeper structure; see Question \ref{ques:cryscohfgq}.

Perhaps more interesting than the theorem above is our method, which relies on Illusie's derived de Rham cohomology \cite{IllusieCCLNM2}. This theory, which in hindsight belongs to derived algebraic geometry, is a refinement of classical de Rham cohomology that works better for singular varieties; the difference, roughly, is the replacement of the cotangent sheaf with the cotangent complex. Theorems from \cite{Bhattpadicddr} show: (a) derived de Rham cohomology agrees with crystalline cohomology  for lci varieties, and (b) derived de Rham cohomology is computed by a ``conjugate'' spectral sequence whose $E_2$-terms come from {\em coherent} cohomology on the Frobenius twist. These results transfer calculations from crystalline cohomology to coherent cohomology, where it is much easier to localise calculations at the singularities (see the proof of Proposition \ref{prop:metaex}).  As a bonus, this method yields a natural (infinite) increasing bounded below exhaustive filtration with finite-dimensional graded pieces on the crystalline cohomology of {\em any} lci proper variety.

\subsection*{Organisation of this note} In \S \ref{sec:ddrsummary}, we review the relevant results from derived de Rham cohomology together with the necessary categorical background. Next, we study (wedge powers of) the cotangent complex of some complete intersections in \S \ref{sec:lcicccomp}. This analysis is used in \S \ref{ss:frobliftex} to provide examples of some singular projective varieties (such as nodal curves, or lci toric varieties) whose crystalline cohomology is always infinitely generated; all these examples admit local lifts to $W_2$ where Frobenius also lifts. Examples which are not obviously liftable (such as ordinary double points in high dimensions) are discussed in \S \ref{ss:odpex}. We conclude with the aforementioned Question \ref{ques:cryscohfgq}.

\subsection*{Notation.}
Let $k$ and $W$ be as above, and set $W_2 = W/p^2$. For a $k$-scheme $X$, let $X^{(1)}$ denote the Frobenius-twist of $X$; we identify the \'etale topology on $X$ and $X^{(1)}$. We use $H^n_\crys(X/k)$ and $H^n_\crys(X/W)$ to denote Berthelot's crystalline cohomology groups relative to $k$ and $W$ respectively. All sheaves are considered with respect to the Zariski topology (unless otherwise specified), and all tensor products are derived. We say that $X$ lifts to $W_2$ compatibly with Frobenius if there exists a flat $W_2$-scheme $\calX$ lifting $X$, and a map $\calX \to \calX$ lifting the Frobenius map on $X$ and lying over the canonical Frobenius lift on $W_2$. For fixed integers $a \leq b \in \Z$, we say that a complex $K$ over some abelian category has amplitude in $[a,b]$ if $H^i(K) = 0$ for $i \notin [a,b] \subset \Z$. A complex $K$ of abelian groups is connected (resp. simply connected) if $H^i(K) = 0$ for $i > 0$ (resp. for $i \geq 0$). An infinitely generated module over a ring is one that is not finitely generated.  All gradings are indexed by $\Z$ unless otherwise specified. If $A$ is a graded ring, then $A(-j)$ is the graded $A$-module defined $A(-j)_i = A_{i-j}$; we set $M(-j) := M \otimes_A A(-j)$ for any graded $A$-complex $M$. We use $\Delta$ for the category of simplices, and $\Ch(\calA)$ for the category of chain complexes over an abelian category $\calA$.

\subsection*{Acknowledgements.} I thank Johan de Jong, Davesh Maulik, and Mircea Musta\c{t}\v{a} for inspiring conversations. In particular, Example \ref{ex:nodalcurve} was discovered in conversation with de Jong and Maulik, and was the genesis of this note. Both Pierre Berthelot and Arthur Ogus had also independently calculated a variant of this example (unpublished), and I thank them for their prompt response to email inquiries.

\section{Review of derived de Rham theory}
\label{sec:ddrsummary}

In this section, we summarise some structure results in derived de Rham theory that will be relevant in the sequel. We begin by recalling in \S \ref{ss:homalgreview} some standard techniques for working with filtrations in the derived category; this provides the language necessary for the work in \cite{Bhattpadicddr} reviewed in \S \ref{ss:ddrsummaryconjfilt}.

\subsection{Some homological algebra}
\label{ss:homalgreview}

In the sequel, we will discuss filtrations on objects of the derived category. To do so in a homotopy-coherent manner, we use the following model structure:

\begin{construction}
\label{cons:modelstr}
Fix a small category $I$, a Grothendieck abelian category $\calB$, and set $\calA = \Fun(I,\calB)$. We endow $\Ch(\calB)$ with the model structure of \cite[Proposition 1.3.5.3]{LurieHA}: the cofibrations are termwise monomorphisms, while weak equivalences are quasi-isomorphisms. The category $\Fun(I,\Ch(\calB)) = \Ch(\Fun(I,\calB)) = \Ch(\calA)$ inherits a projective model structure by \cite[Proposition A.2.8.2]{LurieHT} where the fibrations and weak equivalences are defined termwise. By \cite[Proposition A.2.8.7]{LurieHT}, the pullback $D(\calB) \to D(\calA)$ induced by the constant map $I \to \{1\}$ has a left Quillen adjoint  $D(\calA) \to D(\calB)$ that we call a ``homotopy-colimit over $I$''.  In fact, exactly the same reasoning shows: given a map $\phi:I \to J$ of small categories, the pullback $\phi^*:D(\Fun(J,\calB)) \to D(\Fun(I,\calB))$ induced by composition with $\phi$ has a left Quillen adjoint $\phi_!:D(\Fun(I,\calB)) \to D(\Fun(J,\calB))$ if $\Ch(\Fun(I,\calB))$ and $\Ch(\Fun(J,\calB))$ are given the projective model structures as above; we often refer to $\phi_!$ as a ``homotopy-colimit along fibres of $\phi$.'' The most relevant examples of $\phi$ for us are: the projections $\Delta^\opp \to \{1\}$, $\Delta^\opp \times \N \to \N$ and $\N \to \{1\}$.
\end{construction}

Using Construction \ref{cons:modelstr}, we can talk about increasing filtrations on objects of derived categories.

\begin{construction}
\label{cons:filtrationsdercat}
Let $\calB$ be a Grothendieck abelian category, and let $\calA := \Fun(\N,\calB)$, where $\N$ is the category associated to the poset $\N$ with respect to the usual ordering. There is a homotopy-colimit functor $F:D(\calA) \to D(\calB)$ which is left Quillen adjoint to the pullback $D(\calB) \to D(\calA)$ induced by the constant map $\N \to \{1\}$; we informally refer to an object $K \in D(\calA)$ as an increasing (or $\N$-indexed) exhaustive filtration on the object $F(K) \in D(\calB)$. There are also restriction functors $[n]^*:D(\calA) \to D(\calB)$ for each $n \in \N$, and maps $[n]^* \to [m]^*$ for $n \leq m$ coherently compatible with composition. For each $n \in \N$, the cone construction defines a functor $\gr_n:D(\calA) \to D(\calB)$ and an exact triangle $[n-1]^* \to [n]^* \to \gr_n$ of functors $D(\calA) \to D(\calB)$; for a filtered object $K \in D(\calB)$, we often use $\gr_n(K)$ to denote $\gr_n$ applied to the specified lift of $K$ to $D(\calA)$. A map $K_1 \to K_2$ in $D(\calA)$ is an equivalence if and only if $[n]^* K_1 \to [n]^* K_2$ is so for all $n \in \N$ if and only if $\gr_n(K_1) \to \gr_n(K_2)$ is so for all $n \in \N$. Given a cochain complex $K$ over $\calB$, the association $n \mapsto \tau_{\leq n} K$ defines an object of $D(\calA)$ lifting the image of $K \in D(\calB)$ under $F$. 
\end{construction}

\begin{remark}
The ``cone construction'' used in Construction \ref{cons:filtrationsdercat} to define $\gr_n$ needs clarification: there is no functor  $\Fun([0 \to 1],D(\calB)) \to D(\calB)$ which incarnates the chain-level construction of the cone. However, the same construction {\em does} define a functor $D(\Fun([0 \to 1],\calB)) \to D(\calB)$, which suffices for the above application (as there are restriction functors $D(\calA) \to D(\Fun([0 \to 1],\calB))$ for each map $[0 \to 1] \to \N$ in $\N$).
\end{remark}

\subsection{The derived de Rham complex and the conjugate filtration}
\label{ss:ddrsummaryconjfilt}

We first recall the definition:

\begin{construction}
For a morphism $f:X \to S$ of schemes, following \cite[\S VIII.2]{IllusieCCLNM2}, the derived de Rham complex $\dR_{X/S} \in \Ch(\Mod_{f^{-1} \calO_S})$ is defined as the homotopy-colimit over $\Delta^\opp$ of the simplicial cochain complex $\Omega^*_{P_\bullet/f^{-1} \calO_S} \in \Fun(\Delta^\opp,\Ch(\Mod_{f^{-1} \calO_S}))$, where $P_\bullet$ is a simplicial free $f^{-1} \calO_S$-algebra resolution of $\calO_X$. When $S$ is an $\F_p$-scheme, the de Rham differential is linear over the $p^{\mathrm{th}}$-powers, so $\dR_{X/S}$ can be viewed as an object of $\Ch(\Mod_{\calO_{X^{(1)}}})$, where $X^{(1)} = X \times_{S,\Frob} S$ is the (derived) Frobenius-twist of $X$ (which is the usual one if $f$ is flat).
\end{construction} 

The following theorem summarises the relevant results from \cite{Bhattpadicddr} about this construction:

\begin{theorem}
\label{thm:ddrsummary}
Let $X$ be a $k$-scheme. Then:
\begin{enumerate}
\item The complex $\dR_{X/k} \in \Ch(\Mod_{\calO_{X^{(1)}}})$ comes equipped with a canonical increasing bounded below separated exhaustive filtration $\Fil^\conj_\bullet$ called the {\em conjugate} filtration. The graded pieces are computed by
\[ \Cartier_i:\gr^\conj_i(\dR_{X/k}) \simeq \wedge^i L_{X^{(1)}/k}[-i].\]
In particular, if $X$ is lci, then $\Fil^\conj_i(\dR_{X/k})$ is a perfect $\calO_{X^{(1)}}$-complex for all $i$.
\item The formation of $\dR_{X/k}$ and the conjugate filtration commutes with \'etale localisation on $X^{(1)}$.
\item There exists a canonical morphism
\[ \R\Gamma(X^{(1)},\dR_{X/k}) \to \R\Gamma_\crys(X/k,\calO)\]
that is an isomorphism when $X$ is an lci $k$-scheme.
\item If there is a lift of $X$ to $W_2$ together with a compatible lift of Frobenius, then the conjugate filtration is split, i.e., there is an isomorphism
\[ \oplus_{i \geq 0} \wedge^i L_{X^{(1)}/k}[-i] \simeq \dR_{X/k}\]
whose restriction to the $i^\mathrm{th}$ summand on the left splits $\Cartier_i$.
\end{enumerate}
\end{theorem}

\begin{remark}
Theorem \ref{thm:ddrsummary} can be regarded as an analogue of the results of Cartier (as explained in \cite{DeligneIllusie}, say) and Berthelot \cite{Berthelot} to the singular case. In particular, when $X$ is quasi-compact, quasi-separated and lci, parts (1) and (3) of Theorem \ref{thm:ddrsummary} together with the end of Remark \ref{rmk:ddrcolimits} yield  a ``conjugate'' spectral sequence
\[ E_2^{p,q}: H^p(X^{(1)},\wedge^q L_{X^{(1)}/k}) \Rightarrow H^{p+q}_\crys(X/k).\]
In the sequel, instead of using this spectral sequence, we will directly use the filtration on $\dR_{X/k}$ and the associated exact triangles; this simplifies bookkeeping of indices.
\end{remark}

\begin{remark}
\label{rmk:ddrcolimits}
We explain the interpretation of Theorem \ref{thm:ddrsummary} using the language of \S \ref{ss:homalgreview}. Let $\calB = \Mod(\calO_{X^{(1)}})$, and let $\calA = \Fun(\N,\calB)$. The construction of the derived de Rham complex $\dR_{X/k} \in D(\calB)$ naturally lifts to an object $\calE \in D(\calA)$ under $F$: if $P_\bullet \to \calO_X$ is the canonical free $k$-algebra resolution of $\calO_X$, then $\Omega^*_{P_\bullet/k} \otimes_{P_\bullet^{(1)}} \calO_{X^{(1)}}$ defines an object of $D(\Fun(\Delta^\opp \times \N,\calB))$ via $(m,n) \mapsto \big(\tau_{\leq n} \Omega^*_{P_m/k} \big) \otimes_{P_m^{(1)}} \calO_{X^{(1)}}$, and its homotopy-colimit over $\Delta^\opp$ (i.e., its pushforward along $D(\Fun(\Delta^\opp \times \N,\calB)) \to D(\Fun(\N,\calB))$) defines the desired object $\calE \in D(\calA)$. This construction satisfies $[n]^* \calE \simeq \Fil^\conj_n(\dR_{X/k})$, so $\gr_n(\calE) \simeq \gr^\conj_n(\dR_{X/k})$ for all $n \in \N$. This lift $\calE \in D(\calA)$ of $\dR_{X/k} \in D(\calB)$ is implicit in any discussion of the conjugate filtration on $\dR_{X/k}$ in this note (as in Theorem \ref{thm:ddrsummary} (1), for example). In the sequel, we abuse notation to let $\dR_{X/k}$ also denote $\calE \in D(\calA)$. When $X$ is quasi-compact and quasi-separated, cohomology commutes with filtered colimits, so $\R\Gamma(X^{(1)},\dR_{X/k}) \simeq \colim_n \R\Gamma(X^{(1)},\Fil^\conj_n(\dR_{X/k}))$. In particular, when restricted to proper varieties, derived de Rham cohomology can be written as a filtered colimit of (complexes of) finite dimensional vector spaces {\em functorially} in $X$.
\end{remark}

\section{Some facts about local complete intersections}
\label{sec:lcicccomp}

In order to apply Theorem \ref{thm:ddrsummary} to compute crystalline cohomology, we need good control on (wedge powers of) the cotangent complex of an lci singularity. The following lemma collects most of the results we will use in \S \ref{ss:frobliftex}.

\begin{lemma}
\label{lem:wedgencc}
Let $(A,\fram)$ be an essentially finitely presented local $k$-algebra with an isolated lci singularity at $\{\fram\}$. Let $N = \dim_{k}(\fram/\fram^2)$ be the embedding dimension. Then:
\begin{enumerate}
\item $\wedge^n L_{A/k}$ is a perfect complex for all $n$. For $n \geq N$, $\wedge^n L_{A/k}$ can be represented by a complex of finite free $A$-modules lying between cohomological degrees $-n$ and $-n + N$ with differentials that are $0$ modulo $\fram$.
\item For any $n \geq N$, the complex $\wedge^n L_{A/k}$ has finite length cohomology groups.
\item For any $n > N$, the group $H^{-n+N}(\wedge^n L_{A/k})$ is non-zero.
\item For any $n > N$, there exists an integer $0 < i \leq N$ such that $H^{-n + N -i}(\wedge^n L_{A/k})$ is non-zero. 
\item If $\dim(A) > 0$ and $n > N$,  then $H^{-n}(\wedge^n L_{A/k}) = 0$, so the integer $i$ in (4) is strictly less than $N$.
\end{enumerate}
\end{lemma}
\begin{proof}
Choose a polynomial algebra $P = k[x_1,\dots,x_N]$ and a map $P \to A$ such that $\Omega^1_{P/k} \otimes_P A \to \Omega^1_{A/k}$ is surjective. By comparing dimensions, the induced map $\Omega^1_{P/k} \otimes_P A \otimes_A A/\fram \to \Omega^1_{A/k} \otimes_A A/\fram$ is an isomorphism. Now consider the exact triangle
\[ L_{A/P}[-1] \to \Omega^1_{P/k} \otimes_P A \to L_{A/k}. \]
The lci assumption on $A$ and the choice of $P$ ensure that $L_{A/P}[-1]$ is a free $A$-module of some rank $r$. Since $\Spec(A)$ is singular at $\fram$, we must have $r > 0$. The previous triangle then induces a (non-canonical) equivalence 
\[ \Big(A^{\oplus r} \stackrel{T}{\to} A^{\oplus N} \Big) \simeq L_{A/k}.\]
The map $T$ must be $0$ modulo $\fram$ as $A^{\oplus N} \to L_{A/k}$ induces an isomorphism on $H^0$ after reduction modulo $\fram$, so  the above presentation yields an identification
\[ L_{A/k} \otimes_A k \simeq \big(k^{\oplus r}[1]\big) \oplus k^{\oplus N}.\]
Computing wedge powers gives
\begin{equation}
\label{eq:wedgencc}\tag{*}
\wedge^n(L_{A/k}) \otimes_A k \simeq \oplus_{a = 0}^N \Big(\wedge^a(k^{\oplus N}) \otimes \Gamma^{n-a}(k^{\oplus r})\Big) [n-a],
\end{equation}
where $\Gamma^*$ is the divided-power functor; here we use that $\wedge^n(V[1]) = \Gamma^n(V)[n]$ for a flat $k$-module $V$ over a ring $k$ (see \cite[\S 7]{QuillenCRCNotes}). We now show the desired claims:
\begin{enumerate}
\item The perfectness of $\wedge^n L_{A/k}$ follows from the perfectness of $L_{A/k}$. The desired representative complex can be constructed as a Koszul complex on the map $T$ above (see the proof of Lemma \ref{lem:cchomoghyp} (4) below); all differentials will be $0$ modulo $\fram$ by functoriality since $T$ is so.
\item We must show that $\big(\wedge^n L_{A/k}\big)_\frap = 0$ for any $\frap \in \Spec(A) - \{\fram\}$ and $n \geq N$. The functor $\wedge^n L_{-/k}$ commutes with localisation, so we must show that $\wedge^n L_{A_{\frap}/k} = 0$ for $\frap$ and $n$ as before, but this is clear: $A_{\frap}$ is the localisation of smooth $k$-algebra of dimension $\leq \dim(A) < N$ for any such $\frap$.
\item By (1),  $H^{-n + N}(\wedge^n L_{A/k}) = 0$ if and only if $\wedge^n L_{A/k}$ has amplitude in $[-n,-n + N-1]$. However, in the latter situation, the complex $\wedge^n L_{A/k} \otimes_A k$ would have no cohomology in degree $-n + N$, contradicting formula \eqref{eq:wedgencc}; note that $r \geq 1$ by the assumption that $\Spec(A)$ is singular at $\fram$.
\item Assume the assertion of the claim is false. Then (3) shows that $\wedge^n L_{A/k}$ is concentrated in a single degree, so $\wedge^n L_{A/k} \simeq M[-n + N]$ for some finite length $A$-module $M$. By (1), $M$ has finite projective dimension. The Auslander-Buschbaum formula and the fact that $A$ is Cohen-Macaulay then show that the projective dimension of $M$ is actually $\dim(A)$. Hence, $M \otimes_A k$ has at most $\dim(A) + 1$ non-zero cohomology groups. On the other hand, formula \eqref{eq:wedgencc} shows that $\wedge^n L_{A/k} \otimes_A k$ has $N+1$ distinct cohomology groups. Hence, $N \leq \dim(A)$, which contradicts the assumption that $\Spec(A)$ is singular at $\fram$.
\item Set $M := H^{-n}(\wedge^n L_{A/k})$, and assume $M \neq 0$. Then $M$ has finite length by (2), and occurs as the kernel of a map of free $A$-modules by (1). Non-zero finite length $R$-modules cannot be found inside free $R$-modules for any $S_1$-ring $R$ of positive dimension, which is a contradiction since complete intersections are $S_1$. \qedhere
\end{enumerate}
\end{proof}

\begin{remark}
The assumption $\dim(A) > 0$ is necessary in Lemma \ref{lem:wedgencc} (5). For example, set $A = k[\epsilon]/(\epsilon^p)$. Then $N = \dim_k(\fram/\fram^2) = 1$, and $L_{A/k} \simeq A[1] \oplus A$. Applying $\wedge^n$ for $n > 0$, we get
\[ \wedge^n(L_{A/k}) \simeq \Gamma^n(A)[n] \oplus \Gamma^{n-1}(A)[n-1],\]
which certainly has non-zero cohomology in degree $-n$.
\end{remark}

Using Lemma \ref{lem:wedgencc}, we can show that the crystalline cohomology of an isolated lci singularity is infinitely generated in a very strong sense:

\begin{corollary}
\label{cor:cryscohlcising}
Let $(A,\fram)$ be as in Lemma \ref{lem:wedgencc}. Assume that $A$ admits a lift to $W_2$ compatible with Frobenius. Then 
\begin{enumerate}
\item $H^i_\crys(\Spec(A)/k) \simeq \oplus_{j \geq 0} H^0(\Spec(A)^{(1)},\wedge^j L_{A^{(1)}/k}[i-j])$ for all $i$.
\item $H^N_\crys(\Spec(A)/k)$ is infinitely generated as an $A^{(1)}$-module.
\end{enumerate}
\end{corollary}
\begin{proof}
Note that $H^*_\crys(\Spec(A)/k)$ is an $A^{(1)}$-module since any divided-power thickening of $A$ is an $A^{(1)}$-algebra.
\begin{enumerate}
\item This follows from Theorem \ref{thm:ddrsummary} (4) and the vanishing of higher quasi-coherent sheaf cohomology on affines.
\item This follows from Lemma \ref{lem:wedgencc} (3). \qedhere
\end{enumerate}
\end{proof}

\begin{remark}
\label{rmk:localinfiniteness}
Let us explain the phrase ``strong sense'' appearing before Corollary \ref{cor:cryscohlcising}. If $A$ is an essentially smooth $k$-algebra, then $H^*_\crys(\Spec(A)/k)$ is infinitely generated over $k$, but not over $A^{(1)}$: the Cartier isomorphism shows that $H^i_\crys(\Spec(A)/k) \simeq \Omega^i_{A^{(1)}/k}$, which is a finite (and even locally free) $A^{(1)}$-module. It is this latter finiteness that also breaks down in the singular setting of Corollary \ref{cor:cryscohlcising}.
\end{remark}

We also record a more precise result on the wedge powers of the cotangent complex for the special case of the co-ordinate ring of a smooth hypersurface; this will be used in \S \ref{ss:odpex}. 

\begin{lemma}
\label{lem:cchomoghyp}
Let $A$ be the localisation at $0$ of $k[x_0,\dots,x_N]/(f)$, where $f$ is a homogeneous degree $d$ polynomial defining a smooth hypersurface in $\P^N$. Assume $p \nmid d$. Then
\begin{enumerate}
\item $A$ is graded.
\item The quotient $M = A/(\frac{\partial f}{\partial x_0},\dots,\frac{\partial f}{\partial x_N})$ is a finite length graded $A$-module whose non-zero weights $j$ are contained in the interval $0 \leq j \leq (d-2)(N+1)$.
\item The $A$-linear Koszul complex $K := K_A(\{\frac{\partial f}{\partial x_i}\})$ of the sequence of partials is equivalent to $M \oplus M(-d)[1]$ as a graded $A$-complex.
\item For $n > N$, we have an equivalence of graded $A$-complexes
\[ \wedge^n L_{A/k}[-n] \simeq M\big( (N+1)(d-1) - nd \big)[-N - 1] \oplus M\big( (N+1)(d-1) - nd - d \big)[-N].\]
\item Assume that $N$ and $d$ satisfy $N(d-2) < d+2$. Fix $j$ and $n$ with $N < j < n$.  Then all graded $k$-linear maps
\[ \wedge^n L_{A/k}[-n] \to \wedge^{j} L_{A/k}[-j][1] \]
are nullhomotopic as graded $k$-linear maps.
\end{enumerate}
\end{lemma}

\begin{proof}
Let $S = k[x_0,\dots,x_N]$ denote the polynomial ring. We note first the assumption $p \nmid d$ implies (by the Euler relation) that $f$ lies in the ideal $J(f) \subset S$ generated by the sequence $\{\frac{\partial f}{\partial x_i}\}$ of partials. Since $f$ defines a smooth hypersurface, the preceding sequence cuts out a zero-dimensional scheme in $S$, and hence must be a regular sequence by Auslander-Buschbaum.  In particular, each $\frac{\partial f}{\partial x_i}$ is non-zero of degree $d-1$. We now prove the claims:
\begin{enumerate}
\item This is clear.
\item Since $f \in J(f)$, the quotient $M$ is identified with $S/J(f)$, so the claim follows from \cite[Corollary 6.20]{VoisinHodgeII}.
\item Consider the $S$-linear Koszul complex $L := K_S(\{\frac{\partial f}{\partial x_i}\})$ of the sequence of partials. Since the partials span a regular sequence in $S$, we have an equivalence $L \simeq S/J(f) \simeq A/J(f)\simeq M$ of graded $S$-modules. Now the complex $K$ is simply $L \otimes_S A \simeq M \otimes_S A$. Since $M$ is already an $A$-module, we get an identification $K \simeq M \otimes_A (A \otimes_S A)$ as graded $A$-modules, where the right hand side is given the $A$-module structure from the last factor. The resolution $\Big( S(-d) \stackrel{f}{\to} S \Big) \simeq A$ then shows that $K \simeq M \otimes_A\Big(A(-d) \stackrel{0}{\to} A\Big) \simeq M \oplus M(-d)[1]$.
\item Set $L = (f)/(f^2)$, $E = \Omega^1_{S/k} \otimes_S A$, and $c:L \to E$ to be the map defined by differentiation. Then the two-term complex defined by $c$ is identified with $L_{A/k}$. Taking wedge powers for $n > N$ then shows (see \cite[Corollary 1.2.7]{KatoSaito}, for example) that the complex
\begin{equation}
\label{eq:pdkoscomp}\tag{**}
 \Gamma^n(L) \otimes_A \wedge^0(E) \to \Gamma^{n-1}(L) \otimes_A \wedge^1(E) \to \dots \to \Gamma^{n-(N+1)}(L) \otimes_A \wedge^{N+1}(E)
\end{equation}
computes $\wedge^n L_{A/k}[-n]$; here the term on the left is placed in degree $0$. Explicitly, the differential
\[ \Gamma^i(L) \otimes_A \wedge^k(E) \to \Gamma^{i-1}(L) \otimes \wedge^{k+1}(E) \]
is given by 
\[ \gamma_i(f) \otimes \omega \mapsto \gamma_{i-1}(f) \otimes \big(c(f) \wedge \omega\big) = (-1)^k \cdot \gamma_{i-1}(f) \otimes \big(\omega \wedge df\big).\] 
In particular, if we trivialise $\Gamma^i(L)$ using $\gamma_i(f)$, then this differential is identified with left-multiplication by $df$ in the exterior algebra $\wedge^*(E)$. We leave it to the reader to check that the complex \eqref{eq:pdkoscomp} above is isomorphic to $K\big( (N+1)(d-1) - nd \big)[-N - 1]$; the rest follows from (3).

\item Let $M' = M\big((N+1)(d-1)\big)[-N-1]$. Then $M'$ is, up to a shift, a graded $A$-module whose weights lie in an interval size of $(d-2)(N+1)$ by $(2)$. By (4), we have
\[ \wedge^n L_{A/k}[-n] \simeq M'(-nd) \oplus M'(-nd - d)[1]\]
and
\[ \wedge^j L_{A/k}[-j][1] \simeq M'(-jd)[1] \oplus M'(-jd - d)[2].\]
Thus, we must check that all graded $k$-linear maps $M'(-nd-d) \to M'(-jd)$ are nullhomotopic. Twisting, it suffices to show $M'$ and $M'( (n+1-j) d)$ do not share a weight. If they did, then $(n+1-j) d \leq (d-2)(N+1)$. Since $j < n$, this implies $2d \leq (d-2)(N+1)$, i.e., $d+2 \leq N(d-2)$, which contradicts the assumption. \qedhere
\end{enumerate}
\end{proof}

\begin{remark}
\label{rmk:degreeslowdegcone}
The assumption $N(d-2) < d+2$ in Lemma \ref{lem:cchomoghyp} (5) is satisfied in exactly the following cases: $N \geq 5$ with $d = 2$, $N = 3, 4$ with $d \leq 3$, $N = 2$ with $d \leq 5$, and $N = 1$ with any $d \geq 1$. In particular, an ordinary double point of any dimension satisfies the assumptions of Lemma \ref{lem:cchomoghyp} in any odd characteristic. We also remark that in this case (i.e., when $d = 2$), the proof of Lemma \ref{lem:cchomoghyp} (5) shows that the space of graded $k$-linear maps $\wedge^n L_{A/k}[-n] \to \wedge^{j} L_{A/k}[-j][1]$ is simply connected.
\end{remark}

\begin{remark}
Lemma \ref{lem:cchomoghyp} (5) only refers to space of graded $k$-linear maps $\wedge^n L_{A/k}[-n] \to \wedge^{j} L_{A/k}[-j][1]$, and not the space of such graded $A$-linear maps. In particular, it can happen that a graded $A$-linear map $\wedge^n L_{A/k}[-n] \to \wedge^j L_{A/k}[-j][1]$ is nullhomotopic as a graded $k$-linear map, but not as an $A$-linear map. 
\end{remark}

Theorem \ref{thm:ddrsummary} will be used to control on the mod $p$ crystalline cohomology of an lci $k$-scheme. To lift these results to $W$, we will use the following base change isomorphism; see \cite{BhattdeJong} for more details.

\begin{lemma}
\label{lem:cryscohbasechange}
Let $X$ be a finite type lci $k$-scheme. Then the $W$-complex $\R\Gamma_\crys(X/W)$ has finite amplitude, and there is a base change isomorphism
\[ \R\Gamma_\crys(X/W) \otimes_W k \simeq \R\Gamma_\crys(X/k).\]
\end{lemma}
\begin{proof}
By a Mayer-Vietoris argument, we immediately reduce to the case where $X = \Spec(A)$ is affine. In this case, let $D$ be the $p$-adic completion of the divided power envelope of a surjection $P \to A$ from a finite type polynomial $W$-algebra $P$. Then $\R\Gamma_\crys(X/W)$ is computed by
\[ \Omega^*_{P/W} \otimes_P D. \]
Since this complex has finite amplitude, the first claim is proven. Next, if $P_0 = P/p$, and $D_0$ is the divided power envelope of $P_0 \to A$, then $\R\Gamma_\crys(X/k)$ is computed by
\[ \Omega^*_{P_0/k} \otimes_{P_0} D_0.\]
The claim now follows from the well-known fact that $D$ is $W$-flat (since $A$ is lci), and $D \otimes_W k \simeq D_0$.
\end{proof}

\section{Examples}
\label{sec:ex}

We come to the main topic of this note: examples of singular proper lci $k$-varieties with large crystalline cohomology.  In \S \ref{ss:frobliftex}, using lifts of Frobenius, we show that certain singular proper varieties (such as nodal curves, or singular lci toric varieties) have infinitely generated crystalline cohomology. In \S \ref{ss:odpex}, we show that a single ordinary double point (or worse) on an lci proper variety forces crystalline cohomology to be infinitely generated.

\subsection{Frobenius-liftable examples}
\label{ss:frobliftex}

We start with a general proposition which informally says: a proper lci $k$-variety has large crystalline cohomology if it contains an isolated singular point whose \'etale local ring lifts to $W_2$ compatibly with Frobenius. Note that lci $k$-algebras always lift to $W_2$, so this is really a condition on Frobenius.

\begin{proposition}
\label{prop:metaex}
Let $X$ be proper lci $k$-scheme. Assume:
\begin{enumerate}
\item There is a closed point $x \in X$ that is an isolated singular point (but there could be other singularities on $X$).
\item There is a lift to $W_2$ of the Frobenius endomorphism of the henselian ring $\calO_{X,x}^h$.
\end{enumerate}
Set $N = \dim_k(\fram_{x}/\fram_x^2)$. Then there exists an integer $0 < i \leq N$ such that:
\begin{enumerate}
\item $H^N_\crys(X/k)$ is infinitely generated over $k$.
\item $H^{N-i}_\crys(X/k)$ is infinitely generated over $k$.
\item At least one of $H^{N+1}_\crys(X/W)[p]$ and $H^N_\crys(X/W)/p$ is infinitely generated over $k$.
\item At least one of $H^{N+1-i}_\crys(X/W)[p]$ and $H^{N-i}_\crys(X/W)/p$ is infinitely generated over $k$.
\end{enumerate}
If $\dim(\calO_{X,x}) > 0$, then the integer $i$ above can be chosen to be strictly less than $N$.
\end{proposition}
\begin{proof}
The desired integer $i$  will be found in the proof of (2) below.
\begin{enumerate}
\item Consider the exact triangle
\[ \Fil^\conj_N(\dR_{X/k}) \to \dR_{X/k} \to \calQ \]
in the $\calO_{X^{(1)}}$-complexes, where $\calQ$ is defined as the homotopy-cokernel. Theorem \ref{thm:ddrsummary} (1) and the lci assumption on $X$ show that $\Fil^\conj_N(\dR_{X/k})$ is a perfect complex on $X^{(1)}$, so $H^i(X^{(1)},\Fil^\conj_N(\dR_{X/k}))$ is a finite dimensional vector space for all $i$ by properness. By Theorem \ref{thm:ddrsummary} (3), to show that $H^N_\crys(X/k)$ is infinitely generated, it suffices to show that $H^N(X^{(1)},\calQ)$ is an infinite dimensional $k$ vector space. First, we show:

\begin{claim}
The natural map $\R\Gamma(X,\calQ) \to \calQ_x$ is a projection onto a summand as $k$-complexes.
\end{claim}
\begin{proof}[Proof of Claim]
Let $j:U \to X$ be an affine open neighbourhood of $x$ such that $U$ has an isolated singularity at $x$, and let $j':V = X - \{x\} \subset X$. By Theorem \ref{thm:ddrsummary} (2), $\calQ|_{U \cap V} \simeq 0$ since $U \cap V$ is $k$-smooth. Hence, the Mayer-Vietoris sequence for the cover $\{U,V\}$ of $X$ and the complex $\calQ$ degenerates to show
\[ \R\Gamma(X,\calQ) \simeq \R\Gamma(U,\calQ) \oplus \R\Gamma(V,\calQ).\]
It now suffices to show that $\R\Gamma(U,\calQ) \simeq \calQ_x$. By Theorem \ref{thm:ddrsummary} (1), $\calQ|_U$ admits an increasing bounded below separated exhaustive filtration with graded pieces $\wedge^n L_{U/k}[-k]$ for $n > N$. Since cohomology commutes with filtered colimits (as $U$ is affine), $\R\Gamma(U,\calQ)$ also inherits such a filtration with graded pieces computed by $\R\Gamma(U,\wedge^n L_{U^{(1)}/k}[-k])$ for $n > N$. Applying the same analysis to $\calQ_x$ reduces us to checking that $\R\Gamma(U,\wedge^n L_{U^{(1)}/k}[-n]) \simeq \wedge^n L_{\calO_{X,x}^{(1)}/k}[-n]$ for $n > N$. But this is clear: for $n > N$,  $\wedge^n L_{U^{(1)}/k}[-n]$ is a perfect complex on $U^{(1)}$ that is supported only at $x$ and has stalk $\wedge^n L_{\calO_{X,x}^{(1)}/k}[-n]$.
\end{proof}

To compute the stalk $\calQ_x$, define $\calQ'$ via the exact triangle
\[ \Fil^\conj_N(\dR_{\calO_{X,x}^h/k}) \to \dR_{\calO_{X,x}^h/k} \to \calQ'.\]
Then $\calQ_x = \calQ'$ by Theorem \ref{thm:ddrsummary} (2), the finite length property of $\calQ_x$, and the fact that $\calO_{X,x}^h \otimes_{\calO_{X,x}} M \simeq M$ for any finite length $\calO_{X,x}$-module $M$. Moreover,  $\calQ'$ can be computed using the Frobenius lifting assumption and Theorem \ref{thm:ddrsummary} (4):
\[ \calQ_x \simeq \calQ' \simeq  \oplus_{n = N+1}^\infty \wedge^n L_{\calO_{X,x}^{(1),h}/k}[-n].\]
Thus, to prove that $H^N(X^{(1)},\calQ)$ is infinitely generated, it suffices to show that $H^N(\calQ_x)$ is infinitely generated. This follows from the formula above  and Lemma \ref{lem:wedgencc} (3).
\item By combining the proof of (1) with Lemma \ref{lem:wedgencc} (4) and the pigeonhole principle, one immediately finds an integer $0 < i \leq N$ such that $H^{N-i}_\crys(X/k)$ is infinitely generated over $k$. Lemma \ref{lem:wedgencc} (5) shows that we can choose such an $i$ with $i < N$ if $\dim(\calO_{X,x}) > 0$.
\item The base change isomorphism from Lemma \ref{lem:cryscohbasechange} gives a short exact sequence
\[ 0 \to H^N_\crys(X/W)/p \to H^N_\crys(X/k) \to H^{N+1}_\crys(X/W)[p] \to 0,\]
so the claim follows from (1).
\item The same argument as (3) works using (2) instead of (1). \qedhere
\end{enumerate}
\end{proof}

We need the following elementary result on Frobenius liftings:

\begin{lemma}
\label{lem:liftfrob}
Let $A$ be a $k$-algebra that admits a lift to $W_2$ together with a compatible lift of Frobenius. Then the same is true for any ind-\'etale $A$-algebra $B$ (such as the henselisation $A$ at a point).
\end{lemma}
\begin{proof}
This follows by deformation theory since $L_{B/A} = 0$ for $B$ as above.
\end{proof}

Specialising Proposition \ref{prop:metaex} leads to the promised examples.

\begin{example}
\label{ex:infinitesimalcurves}
Let $X = \Spec(k[x]/x^n)$ for some $n > 1$. Then $H^1_\crys(X/k)$, $H^0_\crys(X/k)$, $H^1_\crys(X/W)/p$, and $H^1_\crys(X/W)[p]$ are all infinitely generated. To see this, note first that Proposition \ref{prop:metaex} applies directly since $X$ is a proper lci $k$-scheme with a lift of Frobenius to $W_2$. Moreover, since $X$ can be realised as a subscheme of $\A^1$, the only non-zero cohomology groups are $H^1_\crys$ and $H^0_\crys$ (over $W$, as well as over $k$). The rest follows directly from Proposition \ref{prop:metaex} and the observation that $H^0_\crys(X/W) = W$ (which is true for any lci $k$-scheme).
\end{example}

\begin{example}
\label{ex:nodalcurve}
Let $X$ be a proper nodal $k$-curve with at least one node. Then $H^1_\crys(X/k)$ and $H^2_\crys(X/k)$ are infinitely generated. Moreover, $H^2_\crys(X/W)/p$, and at least one of $H^1_\crys(X/W)/p$ and $H^2_\crys(X/W)[p]$, are infinitely generated. Most of these claims follow directly from Proposition \ref{prop:metaex}: a nodal curve is always lci, and the henselian local ring at a node on $X$ is isomorphic to the henselisation of $k[x,y]/(xy)$ at the origin, which is a one-dimensional local ring that admits a lift to $W_2$ compatible with Frobenius by Lemma \ref{lem:liftfrob}. It remains to show that $H^3_\crys(X/W)[p]$ is finitely generated. As pointed out by de Jong, the stronger statement $H^3_\crys(X/W) = 0$ is true. If $u:(X/W)_\crys \to X_\zar$ is the natural map (i.e., $u_*(\calF)(U \subset X) = \Gamma( (U/W)_\crys,\calF|_U)$), then $\R^i u_* \calO_{X/W,\crys}$ is non-zero only for $0 \leq i \leq 2$, and $\R^2 u_* \calO_{X/W,\crys}$ is supported only at the nodes\footnote{Proof sketch: Replace the Zariski topology with the Nisnevich topology in the foundations of crystalline cohomology, and then use that every nodal curve is Nisnevich locally planar. This observation yields a three-term de Rham complex computing the stalks of $\R^i u_* \calO_{X/W,\crys}$.}. The rest follows from the Leray spectral sequence as $X_\zar$ has cohomological dimension $1$.
\end{example}

\begin{example}
\label{ex:toricsing}
Let $X$ be a proper lci $k$-scheme. Assume that $x \in X(k)$ is an isolated  singular point (but there could be other singularities on $X$) such that $\calO_{X,x}^h$ is toric of embedding dimension $N$. Then $H^N_\crys(X/k)$ is infinitely generated, and at least one of $H^N_\crys(X/W)$ and $H^{N+1}_\crys(X/W)[p]$ is infinitely generated over $W$. This follows from Proposition \ref{prop:metaex} and Lemma \ref{lem:liftfrob} since toric rings lift to $W_2$ compatibly with Frobenius (use multiplication by $p$ on the defining monoid). Some specific examples are: any proper toric variety with isolated lci singularities, or any proper singular $k$-scheme of dimension $\leq 3$ with at worst ordinary double points.
\end{example}

\begin{example}
Let $(E,e)$ be an ordinary elliptic curve over $k$, and let $X$ be a proper lci $k$-surface with a singularity at $x \in X(k)$ isomorphic to the one on the affine cone over $E \subset \P^2_k$ embedded via $\calO(3[e])$; for example, we could take $X$ to be the projective cone on $E \subset \P^2_k$. Then $H^3_\crys(X/k)$ and one of $H^2_\crys(X/k)$ or $H^1_\crys(X/k)$ are infinitely generated over $k$. This can be proven using Proposition \ref{prop:metaex} and the theory of Serre-Tate canonical lifts. Since we prove a more general and shaper result in Example \ref{ex:curvecone}, we leave details of this argument to the reader.
\end{example}

\subsection{Conical examples}
\label{ss:odpex}

Our goal here is to show that the presence of an ordinary double point forces crystalline cohomology to be infinitely generated. In fact, more generally, we show the same for any proper lci variety that has a singularity isomorphic to the cone on a low degree smooth hypersurface.   We start with an {\em ad hoc} definition.

\begin{definition}
A local $k$-algebra $A$ is called a {\em low degree cone} if its henselisation is isomorphic to the henselisation at the origin of the ring $k[x_0,\dots,x_N]/(f)$, where $f$ is a homogeneous degree $d$ polynomial defining a smooth hypersurface in $\P^N$ such that $N(d-2) < d+2$. The integer $d$ is called the {\em degree} of this cone; if $d = 2$, we also call $A$ an {\em ordinary double point}. A closed point $x \in X$ on a finite type $k$-scheme $X$ is called {\em low degree conical singularity} (respectively, an {\em ordinary double point}) if $\calO_{X,x}^h$ is a low degree cone (respectively, an ordinary double point).
\end{definition}

We start by showing that the conjugate spectral sequence must eventually degenerate for low degree cones:

\begin{proposition}
\label{prop:odpconjdiff}
Let $A$ be low degree cone of degree $d$. Assume $p \nmid d$. Then for $n > \dim(A)$, the extensions 
\[ \gr_n^\conj(\dR_{A/k}) \to \Fil_{n-1}^\conj(\dR_{A/k})/\Fil_{\dim(A)}^\conj(\dR_{A/k})[1] \] 
occurring in the conjugate filtration are nullhomotopic when viewed as $k$-linear extensions. In particular, there exist $k$-linear isomorphisms
\[ \Fil_{n}^\conj(\dR_{A/k})/\Fil_{\dim(A)}^\conj(\dR_{A/k}) \simeq \oplus_{j = \dim(A) + 1}^{n} \gr^\conj_j(\dR_{A/k}).\]
splitting the conjugate filtration for any $n > \dim(A)$.
\end{proposition}
\begin{proof}
By replacing $A$ with its henselisation and then using the \'etale invariance of cotangent complexes and Theorem \ref{thm:ddrsummary} (2), we may assume $A$ is the localisation of $k[x_0,\dots,x_N]/(f)$ at the origin for some homogeneous degree $d$ polynomial $f$ defining a smooth hypersurface in $\P^N$. In particular, $A$ is graded. Also, by functoriality, the conjugate filtration is compatible with the grading. Recall that the extensions in question arise from the triangles
\[ \Fil_{n-1}^\conj(\dR_{A/k})/\Fil_{\dim(A)}^\conj(\dR_{A/k}) \to \Fil_n^\conj(\dR_{A/k})/\Fil_{\dim(A)}^\conj(\dR_{A/k}) \to \gr_n^\conj(\dR_{A/k}).\]
These triangles (and thus the corresponding extensions) are viewed as living in the derived category of graded $k$-vector spaces. By induction, we have to show the following: assuming a graded splitting
\[ s_{n-1}:\Fil_{n-1}^\conj(\dR_{A/k})/\Fil_{\dim(A)}^\conj(\dR_{A/k}) \simeq \oplus_{j = \dim(A) + 1}^{n-1} \gr^\conj_j(\dR_{A/k})\]
of the conjugate filtration, there exists a graded splitting
\[ s_n:\Fil_{n}^\conj(\dR_{A/k})/\Fil_{\dim(A)}^\conj(\dR_{A/k}) \simeq \oplus_{j = \dim(A) + 1}^{n} \gr^\conj_j(\dR_{A/k}),\]
of the conjugate filtration compatible with $s_{n-1}$. Chasing extensions, it suffices to show: for $\dim(A) < j < n$, all graded maps
\[ \gr^\conj_n(\dR_{A/k}) \to \gr^\conj_j(\dR_{A/k})[1]\]
are nullhomotopic. This comes from Lemma \ref{lem:cchomoghyp} (5) and Theorem \ref{thm:ddrsummary} (1).
\end{proof}

\begin{remark}
An inspection of the proof of Proposition \ref{prop:odpconjdiff} coupled with Remark \ref{rmk:degreeslowdegcone} shows that if $A$ is an ordinary double point, then the isomorphism
\[ \Fil_{n}^\conj(\dR_{A/k})/\Fil_{\dim(A)}^\conj(\dR_{A/k}) \simeq \oplus_{j = \dim(A) + 1}^{n} \gr^\conj_j(\dR_{A/k})\]
is unique, up to non-unique homotopy. We do not know any applications of this uniqueness.
\end{remark}

Using Proposition \ref{prop:odpconjdiff}, we can prove infiniteness of crystalline cohomology for some cones:

\begin{corollary}
\label{cor:lowdegcone}
Let $X$ be a proper lci $k$-scheme. Assume that there is low degree conical singularity at a closed point $x \in X$ with degree $d$ and  embedding dimension $N$.  If $p \nmid d$, then $H^N_\crys(X/k)$ and $H^{N-1}_\crys(X/k)$ are infinitely generated $k$-vector spaces. \end{corollary}
\begin{proof}
We combine the proof strategy of Proposition \ref{prop:metaex} with Proposition \ref{prop:odpconjdiff}. More precisely, following the proof of Proposition \ref{prop:metaex} (1), it suffices to show that $\calQ'$ is infinitely generated when regarded as a complex of $k$-vector spaces.  Now $\calQ'$ admits an increasing bounded below separated exhaustive filtration with graded pieces given by $\gr^\conj_n(\dR_{\calO_{X,x}^h/k})$ for $n > N$. By Proposition \ref{prop:odpconjdiff}, there is a (non-canonical) isomorphism
\[ \calQ'  \simeq \oplus_{n > N} \wedge^n L_{\calO_{X,x}^{(1),h}/k}[-n].\]
The rest follows from Lemma \ref{lem:cchomoghyp} (4) (note that embedding dimension in {\em loc.\ cit.} is $N+1$, so we must shift by $1$).
\end{proof}

We can now give the promised example:

\begin{example}
\label{ex:odp}
Let $X$ be any proper lci variety that contains an ordinary double point $x \in X(k)$ of embedding dimension $N$ with $p$ odd; for example, we could take $X$ to be the projective cone over a smooth quadric in $\P^{N-1}$. Then $H^N_\crys(X/k)$ and $H^{N-1}_\crys(X/k)$ are infinitely generated by Corollary \ref{cor:lowdegcone}. 
\end{example}

All examples given so far have rational singularities, so we record an example that is not even log canonical. 

\begin{example}
\label{ex:curvecone}
Let $X$ be any proper lci surface that contains a closed point $x \in X(k)$ with $\calO_{X,x}^h$ isomorphic to the henselisation at the vertex of the cone over a smooth curve $C \subset \P^2$ of degree $\leq 5$. If $p \geq 7$,  then $H^3_\crys(X/k)$ and $H^{2}_\crys(X/k)$ are infinitely generated by Corollary \ref{cor:lowdegcone}. 
%Note that if the degree of $C \subset \P^2$ is $4$ or $5$, then $\calO_{X,x}^h$ does not admit a lift to $W_2$ compatible with Frobenius. If it did, then so would the affine cone $(S,s)$ on $C \subset \P^2$ (by the local-to-global spectral sequence to relate $\Ext$-groups on $S$ and $\calO_{S,s}^h \simeq \calO_{X,x}^h$), and hence so would $U = S - \{s\}$. 
%The total space of $\G_m$-torsor associated to the line bundle $\calO_{\P^2}(-1)|_C$ would also admit such a lift. By the usual exact sequences, $C$ itself then admits a lift to $W_2$ compatible with Frobenius, but this forces $C$ to be Frobenius-split, which is impossible (look at the action on $H^1(C,\Omega^1_{C/k})$).
\end{example}

\begin{remark}
We do not know whether the ordinary double point from Example \ref{ex:odp} admits a lift to $W_2$ compatible with Frobenius in arbitrary dimensions; similarly for the cones in Example \ref{ex:curvecone} (except for ordinary elliptic curves).
\end{remark}

We end with a question that has not been answered in this note.

\begin{question}
\label{ques:cryscohfgq}
Do there exist singular proper $k$-varieties with finite dimensional crystalline cohomology over $k$? Do there exist singular finite type $k$-algebras $A$ whose crystalline cohomology relative to $k$ is finitely generated over $A^{(1)}$?
\end{question}

\bibliography{my}

\end{document}